\documentclass[runningheads]{llncs}
\usepackage[T1]{fontenc}
\usepackage{graphicx}
\usepackage{color}
\usepackage[english]{babel}
\usepackage{amsmath}
\usepackage{amssymb}
\usepackage{amsfonts}
\usepackage{array}
\usepackage{epsfig}
\usepackage{float}
\usepackage{enumitem}  
\usepackage{epstopdf}
\usepackage{stmaryrd}
\usepackage{standalone}
\usepackage{tikz, subfigure}
\usepackage{yfonts}
\usepackage{xcolor}
\usepackage{mathtools}

\usepackage{booktabs}  
\usepackage{multirow}
\usepackage[mathlines]{lineno}

\usepackage{algorithm}
\usepackage{algorithmic}
\usepackage{cleveref}

\newcommand{\jump}[1]{\ensuremath{[\![#1]\!]} }
\newcommand{\avg}[1]{\ensuremath{\left\{\!\!\left\{#1\right\}\!\!\right\}} }

\newcommand{\half}{{\frac{1}{2}}}
\newcommand{\RR}{{\mathbb{R}}}
\newcommand{\CC}{{\mathbb{C}}}
\renewcommand{\div}{{\operatorname{div}}}
\newcommand{\diag}{{\operatorname{diag}}}

\newcommand{\sfA}{{\mathsf{A}}}
\newcommand{\sfM}{{\mathsf{M}}}
\newcommand{\sfU}{{\mathsf{U}}}
\newcommand{\sfI}{{\mathsf{I}}}
\newcommand{\sfLam}{{\mathsf{\Lambda}}}

\begin{document}
\baselineskip=0.95\normalbaselineskip

\title{Rational approximation preconditioners for multiphysics problems }
\titlerunning{Rational approximation multiphysics preconditioners}

\author{
	Ana Budi\v{s}a\inst{1} \and
	Xiaozhe Hu\inst{2} \and
	Miroslav Kuchta\inst{1} \and
	Kent-Andr\'{e} Mardal\inst{1,3} \and
	Ludmil Zikatanov\inst{4}
}
\authorrunning{A. Budi\v{s}a et al.}

\institute{
	Simula Research Laboratory, 0164 Oslo, Norway \\
	\and
	Tufts University, Medford, MA 02155, USA \\
    \and
	University of Oslo, 0316 Oslo, Norway \and
	Penn State University, University Park, PA 16802, USA\\
}
\maketitle
\begin{abstract}
We consider a class of mathematical models describing multiphysics phenomena interacting through interfaces. On such interfaces, the traces of the fields lie (approximately) in the range of a weighted sum of two fractional differential operators. We use a rational function approximation to precondition such operators. We first demonstrate the robustness of the approximation for ordinary functions given by weighted sums of fractional exponents. Additionally, we present more realistic examples utilizing the proposed preconditioning techniques in interface coupling between Darcy and Stokes equations.

\end{abstract}
\section{Introduction}

Fractional operators arise in the context of preconditioning of coupled multiphysics systems and, in particular, in the problem formulations where the coupling constraint is enforced by a Lagrange multiplier defined on the interface. Examples include the so-called EMI equations in modeling of excitable tissue \cite{tveito2017cell}, reduced order models of microcirculation in 2$d$-1$d$ \cite{kuchta2016preconditioners,lamichhane2004mortar} and 3$d$-1$d$ \cite{kuchta2021analysis,kuchta2019preconditioning} setting or Darcy/Biot--Stokes models \cite{ambartsumyan2018lagrange,layton2002coupling}. We remark that fractional operators have been recently utilized also in monolithic solvers for formulations of Darcy/Biot--Stokes models without the Lagrange multipliers, see \cite{boon2022parameter,boon2022robust}.

The coupling in the multiplier formulations is naturally posed in Sobolev spaces of fractional order. However, for parameter robustness of iterative methods, a more precise setting must be considered, where the interface problem is posed in the intersection (space) of parameter-weighted fractional order Sobolev spaces \cite{holter2020robust}. Here the sums of fractional operators induce the natural inner product. In the examples mentioned earlier, however, the interface preconditioners were realized by eigenvalue decomposition, and thus, while being parameter robust, the resulting solvers do not scale with mesh size. Using Darcy--Stokes system as the canonical example, we aim to show that rational approximations are crucial for designing efficient preconditioners for multiphysics systems.

There have been numerous works on approximating/preconditioning problems such as $\mathcal{D}^s u=f$, e.g. \cite{2015BonitoPasciak-a,2016ChenNochettoOtarolaSalgado-a,2016NochettoOtarolaSalgado-a}. The rational approximation (RA) approach has been advocated, and several techniques based on it were proposed by Hofreither in~\cite{hofreitherUnifiedViewNumerical2020,2021Hofreither-a} (where also a \verb|Python| implementation of Best Uniform Rational Approximation (BURA) for $x^{\lambda}$ is found). Further rational approximations used in preconditioning can be found in~\cite{2020HarizanovLazarovMargenovMarinov-a,2020HarizanovLazarovMargenovMarinovPasciak-a}. An interesting approach, which always leads to real-valued poles and uses Pad\'{e} approximations of suitably constructed time-dependent problems, is found in~\cite{lazarov2020_time}. In our numerical tests, we use Adaptive Antoulas-Anderson (AAA) algorithm~\cite{nakatsukasaAAAAlgorithmRational2018} which is a greedy strategy for locating the interpolation points and then using the barycentric representation of the rational interpolation to define a function that is close to the target. In short, for a given continuous function $f$, the AAA algorithm returns a rational function that approximates the best uniform rational approximation to $f$.  

The rest of the paper is organized as follows. In Section~\ref{sec:darcy-stokes-model}, we introduce a model problem that leads to the sum of fractional powers of a differential operator on an interface acting on weighted Sobolev spaces. In Section~\ref{sec:mixed-FEM}, we introduce the finite element discretizations that we employ for the numerical solution of such problems. Next, Section~\ref{sec:ra} presents some details on the rational approximation and the scaling of the discrete operators. In Section~\ref{sec:results}, we test several relevant scenarios and show the robustness of the rational approximation as well as the efficacy of the preconditioners. Conclusions are drawn in Section~\ref{sec:conclusions}. 

\section{Darcy--Stokes model}\label{sec:darcy-stokes-model}
We study interaction between a porous medium occupying $\Omega_D\subset\mathbb{R}^d$, $d=2, 3$ surrounded by a free flow domain $\Omega_S\supset\Omega_D$ given by a Darcy--Stokes model \cite{layton2002coupling} as: For given volumetric source terms $f_S:\Omega_S\rightarrow\mathbb{R}^d$ and $f_D:\Omega_D\rightarrow\mathbb{R}$ find the Stokes velocity and pressure $u_S:\Omega_S\rightarrow\mathbb{R}^d$, $p_S:\Omega_S\rightarrow\mathbb{R}$ and the Darcy flux and pressure $u_D:\Omega_D\rightarrow\mathbb{R}^d$, $p_D:\Omega_D\rightarrow\mathbb{R}$ such that
\begin{equation}\label{eq:darcy_stokes}
  \begin{aligned}
  -\nabla\cdot\sigma(u_S, p_S) = f_S \text{ and } \nabla\cdot u_S &= 0  &\text{ in }\Omega_S, \\
  u_D + K\mu^{-1}\nabla p_D = 0  \text{ and } \nabla\cdot u_D &= f_D  &\text{ in }\Omega_D, \\
  u_S\cdot \nu_s + u_D\cdot \nu_D &= 0  &\text{ on }\Gamma, \\
  -\nu_S\cdot\sigma(u_S, p_s)\cdot\nu_S - p_D &= 0   &\text{ on }\Gamma, \\
  -P_{\nu_S}\left(\sigma(u_S, p_S)\cdot\nu_S\right) - \alpha\mu K^{-1/2}P_{\nu_S}u_S &= 0   &\text{ on }\Gamma.
  \end{aligned}
\end{equation}
Here, $\sigma(u, p):=2\mu\epsilon(u)-pI$ with $\epsilon(u):=\tfrac{1}{2}(\nabla u + \nabla u^{T})$.
Moreover, $\Gamma:=\partial\Omega_D\cap\partial\Omega_S$ is the common interface,
$\nu_S$ and $\nu_D=-\nu_S$ represent the outward unit normal vectors on $\partial\Omega_S$ and $\partial\Omega_D$.  Given a surface with normal vector $\nu$,
$P_{\nu}:=I-\nu\otimes\nu$ denotes a projection to the tangential plane with normal $\nu$. The final three
equations in \eqref{eq:darcy_stokes} represent the coupling conditions on $\Gamma$. Parameters of the model
(which we shall assume to be constant) are viscosity $\mu>0$, permeability $K>0$ and Beavers-Joseph-Saffman
parameter $\alpha>0$. 

Finally, let us decompose the outer boundary $\partial\Omega_S\setminus\Gamma=\Gamma_u\cup\Gamma_\sigma$,
$\lvert\Gamma_i\rvert>0$, $i=u, \sigma$ and introduce the boundary conditions to
close the system \eqref{eq:darcy_stokes}
\begin{equation*}
  \begin{aligned}
    u_S\cdot\nu_S &= 0 \text{ and }
    P_{\nu_S}\left(\sigma(u_S, p_S)\cdot\nu_S\right) = 0&\text{ on }\Gamma_u, \\
    \sigma(u_S, p_S)\cdot\nu_S &= 0&\text{ on }\Gamma_{\sigma}.
  \end{aligned}
\end{equation*}
Thus, $\Gamma_u$ is an impermeable free-slip boundary. We remark
that other conditions, in particular no-slip on $\Gamma_u$, could be considered without introducing additional challenges. However, unlike the tangential component, constraints for the normal component of velocity are easy\footnote{ Conditions on the normal component can be implemented as Dirichlet boundary conditions and enforced by the constructions of the finite element trial and test spaces. The tangential component can be controlled e.g., by the Nitsche method \cite{stenberg1995some} which modifies the discrete problem operator.} to implement in $H(\div)$-conforming discretization schemes considered below.

Letting $V_S\subset H_1(\Omega_S)$, $V_D\subset H(\div, \Omega_D)$,
$Q_S\subset L^2(\Omega_S)$, $Q_D\subset L^2(\Omega_D)$ and $Q\subset H^{1/2}(\Gamma)$ the variational formulation of \eqref{eq:darcy_stokes} seeks to find $w:=(u, p, \lambda)\in W$, $W:=V\times Q\times \Lambda$, $V:=V_S\times V_D$, $Q:=Q_S \times Q_D$ and $u:=(u_S, u_D)$, $p:=(p_S, p_D)$ such that $\mathcal{A}w = L$ in $W'$, the dual space of $W$, where $L$ is the linear functional of the right-hand sides in \eqref{eq:darcy_stokes} and the problem operator $\mathcal{A}$ satisfies
\begin{equation}\label{eq:darcy_stokes_operator}
    \langle \mathcal{A}w, \delta w\rangle = a_S(u_S, v_S) + a_D(u_D, v_D) + b(u, q) + b_{\Gamma}(v, \lambda)+
     b(v, p) + b_{\Gamma}(u, \delta\lambda),
\end{equation}
where $\delta w:=(v, q, \delta\lambda)$, $v:=(v_S, v_D)$, $q:=(q_S, q_D)$ and $\langle\cdot,\cdot\rangle$ denotes a duality pairing between $W$ and $W'$. The bilinear forms in \eqref{eq:darcy_stokes_operator} are defined as 
\begin{equation}\label{eq:darcy_stokes_components}
  \begin{aligned}
    a_S(u_S, v_S) &:= \int_{\Omega_S}2\mu\epsilon(u_S):\epsilon(v_S)\,\mathrm{d}x + \int_{\Gamma}{\alpha\mu}{K^{-1/2}}P_{\nu_S}u_S\cdot P_{\nu_S}v_S\,\mathrm{d}x,\\
    a_D(u_D, v_D) &:= \int_{\Omega_D}{\mu}{K}^{-1}u_D\cdot v_D\,\mathrm{d}x,\\
    b(v, p) &:= -\int_{\Omega_S}p_S\nabla\cdot v_S\,\mathrm{d}x -\int_{\Omega_D}p_D\nabla\cdot v_D\,\mathrm{d}x,\\
    b_{\Gamma}(v, \lambda) &:= \int_{\Gamma}(v_S\cdot \nu_S + v_D\cdot \nu_D)\lambda\,\mathrm{d}s.
  \end{aligned}
\end{equation}
Here, $\lambda$ represents a Lagrange multiplier whose physical meaning is related to the normal component of the traction vector on $\Gamma$, $\lambda:=-\nu_S\cdot\sigma(u_S, p_S)\cdot\nu_S$, see~\cite{layton2002coupling} where also well-posedness of the problem in the space $W$ above is established.

Following \cite{holter2020robust}, parameter-robust preconditioners for Darcy--Stokes operator $\mathcal{A}$
utilize weighted sums of fractional operators on the interface. Specifically, we shall consider the following operator
\begin{equation}\label{eq:darcy_stokes_iface}
  S := {\mu}^{-1}(-\Delta_{\Gamma} + I_{\Gamma})^{-1/2} + {K}{\mu}^{-1}(-\Delta_{\Gamma} + I_{\Gamma})^{1/2},
\end{equation}
where we have used the subscript $\Gamma$ to emphasize that the operators are considered on the interface.
We note that $-\Delta_{\Gamma}$ is singular in our setting as $\Gamma$ is a closed surface and 
adding lower order term $I_{\Gamma}$ thus ensures positivity.

Letting $A_S$ be the operator induced by the bilinear form $a_S$ in \eqref{eq:darcy_stokes_components} we define the Darcy--Stokes preconditioner as follows,
\begin{equation}\label{eq:darcy_stokes_precond}
\mathcal{B} := \diag\left(A_S,\mu {K}^{-1}(I-\nabla\nabla\cdot), {\mu}^{-1}I, {K}{\mu}^{-1}I, S \right)^{-1}.
\end{equation}
Note that the operators $I$ in the pressure blocks of the preconditioner act on different spaces/spatial domains, i.e., $Q_S$ and $Q_D$. We remark that in the context of Darcy--Stokes preconditioning, \cite{harizanov2022rational} consider BURA approximation for a simpler interfacial operator, namely, ${K}{\mu}^{-1}(-\Delta_{\Gamma} + I_{\Gamma})^{1/2}$. However, the preconditioner cannot yield parameter robustness, cf. \cite{holter2020robust}.

\section{Mixed Finite Element Discretization}\label{sec:mixed-FEM}
In order to assess numerically the performance of rational approximation of $S^{-1}$ in \eqref{eq:darcy_stokes_precond}, stable discretization of the Darcy--Stokes system is needed. In addition to parameter variations, here we also wish to show the algorithm's robustness to discretization and, in particular, the construction of the discrete multiplier space. To this end, we require a family of stable finite element discretizations.

Let $k\geq 1$ denote the polynomial degree. For simplicity, to make sure that the Lagrange multiplier fits well with the discretization of both the Stokes and Darcy domain, we employ the same $H(\mbox{div})$ based discretization in both domains. That is, we discretize the Stokes velocity space $V_S$ by Brezzi-Douglas-Marini $\mathbb{BDM}_{k}$ elements \cite{brezzi1985two} over simplicial triangulations $\Omega^h_{S}$ of $\Omega_S$ and likewise, approximations to $V_D$ are constructed with $\mathbb{BDM}_{k}$ elements on $\Omega^h_{D}$. Here, $h$ denotes the characteristic mesh size. Note that by construction, the velocities and fluxes have continuous normal components across the \emph{interior} facets of the respective triangulations. However, on the interface $\Gamma$, we do not impose any continuity between the vector fields. The pressure spaces $Q_S$, $Q_D$ shall be approximated in terms of discontinuous piecewise polynomials of degree $k-1$, $\mathbb{P}^{\text{disc}}_{k-1}$. Finally, the Lagrange multiplier space is constructed by $\mathbb{P}^{\text{disc}}_{k}$ elements on the triangulation $\Gamma^h:=\Gamma^h_S$, with $\Gamma^h_S$ being the trace mesh of $\Omega^h_{S}$ on $\Gamma$. For simplicity, we assume $\Gamma^h_D=\Gamma^h_S$.

Approximation properties of the proposed discretization are demonstrated in \Cref{fig:darcy_stokes_cvrg}. It can be seen that all the quantities converge with order $k$ (or better) in their respective norms. This is particularly the case for the Stokes velocity, where the error is measured in the $H^1$ norm.

Let us make a few remarks about our discretization. First, observe that by using $\mathbb{BDM}$ elements on a global mesh $\Omega^h_{S}\cup\Omega^h_{D}$ the Darcy--Stokes problem \eqref{eq:darcy_stokes} can also be discretized such that the mass conservation condition $u_S\cdot\nu_S+u_D\cdot\nu_D=0$ on $\Gamma$ is enforced by construction, i.e. no Lagrange multiplier is required. Here, $u$ is the global $H(\div)$-conforming vector field, with $u_i:=u|_{\Omega_i}$, $i=S, D$. Second, we note that the chosen discretization of $V_S$ is only $H(\div, \Omega_S)$-conforming. In turn, stabilization of the tangential component of the Stokes velocity is needed, see, e.g., \cite{2014AyusodeDiosBrezziMariniXuZikatanov}, which translates to modification of the bilinear form $a_S$ in \eqref{eq:darcy_stokes_components} as
\begin{equation}\label{eq:sym_grad_BDM}
\begin{split}
a^h_S(u, v) &:= a_S(u, v) - \sum_{e\in F^h_{S}}\int_{e}2\mu \avg{\epsilon(u)}\cdot\jump{P_{\nu_e}v}\,\mathrm{d}s\\
&- \sum_{e\in F^h_{S}}\int_{e}2\mu\avg{\epsilon(v)}\cdot\jump{P_{\nu_e}u}\,\mathrm{d}s
+ \sum_{e\in F^h_{S}}\int_{e}\frac{2\mu\gamma}{h_e}\jump{P_{\nu_e}u}\cdot \jump{P_{\nu_e}v}\,\mathrm{d}s,
\end{split}
\end{equation}
see \cite{hong2016robust}. Here, $F^{h}_S=F(\Omega^h_S)$ is the collection of interior facets $e$ of triangulation $\Omega^h_S$, while $\nu_e$, $h_e$ denote respectively the facet normal and facet diameter. The stabilization parameter $\gamma>0$ has to be chosen large enough to ensure the coercivity of $a^h_S$. However, the value depends on the polynomial degree. Finally, for interior facet $e$ shared by elements $T^{+}$, $T^{-}$ we define the (facet) jump of a vector $v$ as $\jump{v}:=v|_{T^{+}\cap e}-v|_{T^{-}\cap e}$ and the (facet) average of a tensor $\epsilon$ as $\avg{\epsilon}:=\tfrac{1}{2}\left(v|_{T^{+}\cap e}\cdot\nu_e+v|_{T^{-}\cap e}\cdot\nu_e\right)$.

We conclude this section by discussing discretization of the operator $(-\Delta+I)$ needed for realizing the interface preconditioner \eqref{eq:darcy_stokes_iface}. Since, in our case, the multiplier space $\Lambda_h$ is only $L^2$-conforming we adopt the symmetric interior penalty approach \cite{di2011mathematical} so that in turn for $u, v\in \Lambda_h$,
\begin{equation}\label{eq:delta_dg}
 \begin{split}
  \langle (-\Delta+I) u, v\rangle := &\int_{\Gamma}(\nabla u\cdot \nabla v  + u v)\,\mathrm{d}x
- \sum_{e\in F(\Gamma_h)}\int_{e}\avg{\nabla u} \jump{v}\,\mathrm{d}s\\
&- \sum_{e\in F(\Gamma_h)}\int_{e}\avg{\nabla v} \jump{u}\,\mathrm{d}s
+ \sum_{e\in F(\Gamma_h)}\int_{e}\frac{\gamma}{h_e}\jump{u}\jump{v}\,\mathrm{d}s.
\end{split}
\end{equation}
Here, the jump of a scalar $f$ is computed as $\jump{f}:=f|_{T^{+}\cap e}-f|_{T^{-}\cap e}$ and the average of a vector $v$ reads $\avg{v}:=\tfrac{1}{2}\left(v|_{T^{+}\cap e}\cdot\nu_e+v|_{T^{-}\cap e}\cdot\nu_e\right)$. As before, $\gamma>0$ is a suitable stabilization parameter.

\begin{figure}
  \centering
  \includegraphics[width=0.325\textwidth]{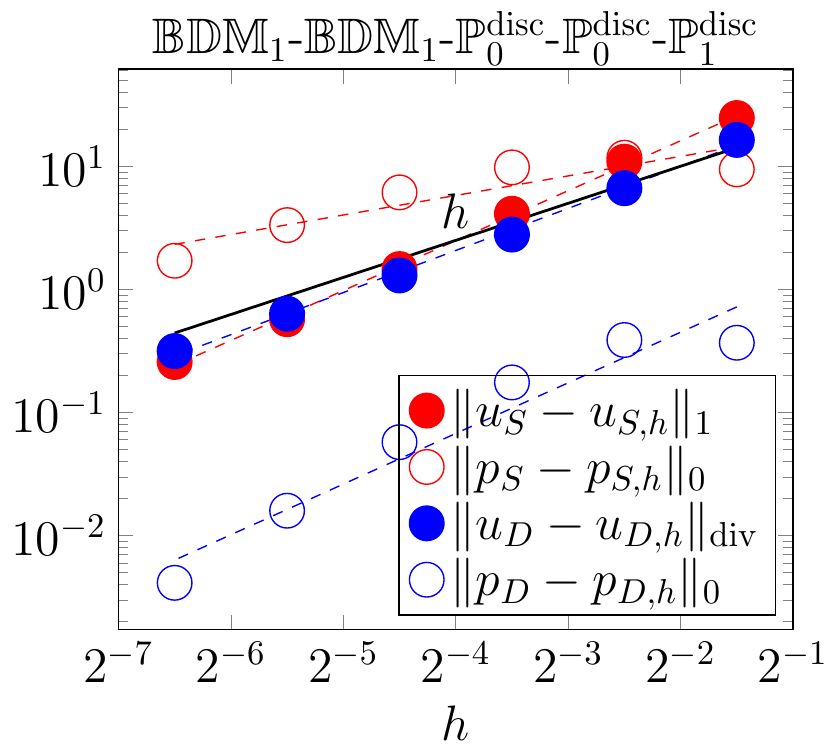}
  \includegraphics[width=0.325\textwidth]{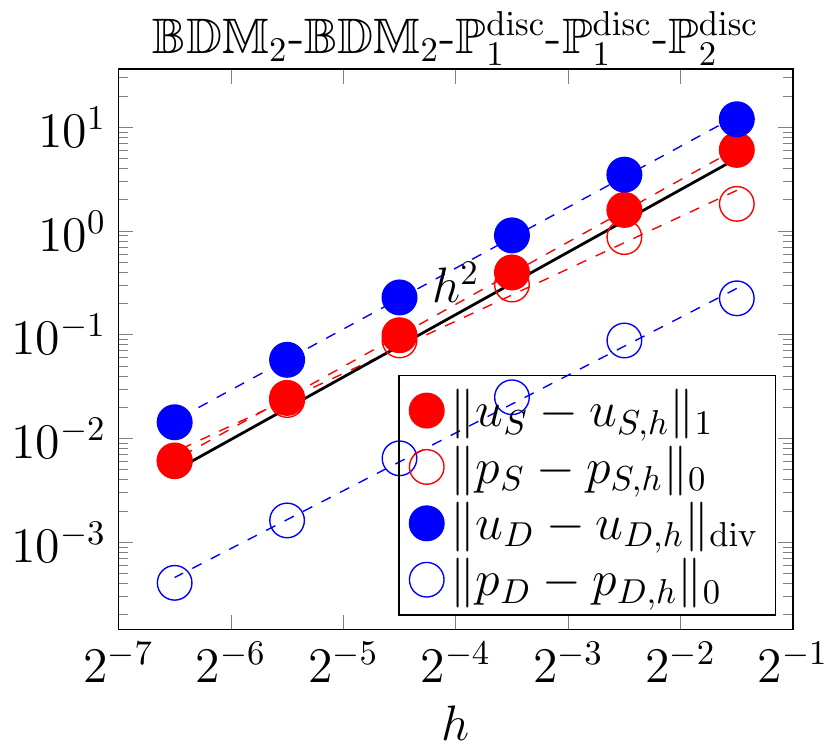}
  \includegraphics[width=0.325\textwidth]{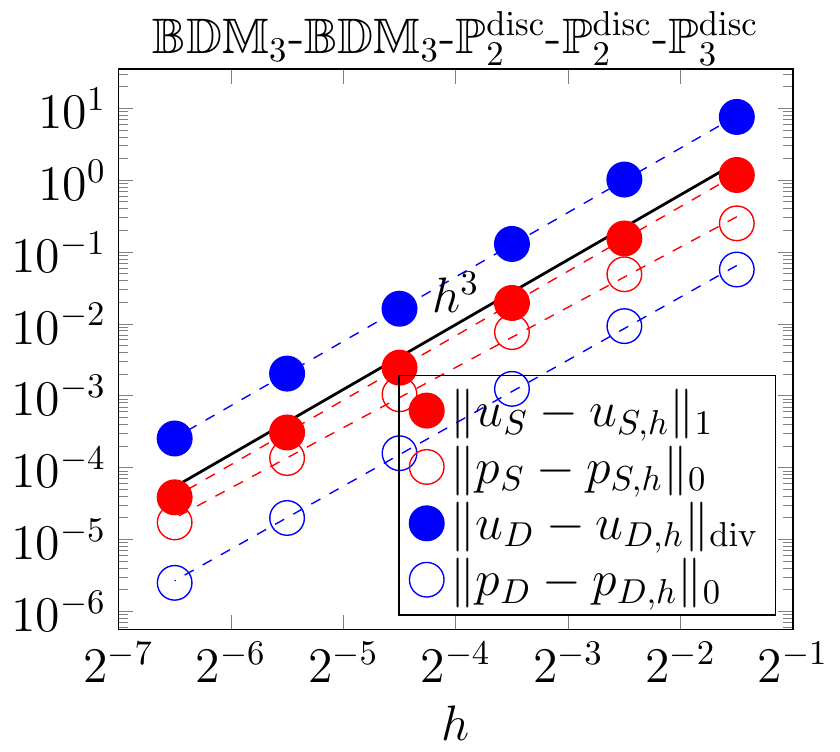}
  \vspace{-10pt}
  \caption{Approximation properties of $\mathbb{BDM}_k$-$\mathbb{BDM}_k$-$\mathbb{P}^{\text{disc}}_{k-1}$-$\mathbb{P}^{\text{disc}}_{k-1}$-$\mathbb{P}^{\text{disc}}_k$ discretization of the Darcy--Stokes problem. Two-dimensional setting is considered with $\Omega_S=(0, 1)^2$ and $\Omega_D=(\tfrac{1}{4}, \tfrac{3}{4})^2$. Parameter values are set as $K=2$, $\mu=3$ and $\alpha=\tfrac{1}{2}$. Left figure is for $k=1$ while $k=2,3$ is shown in the middle and right figures,  respectively. 
  }
  \label{fig:darcy_stokes_cvrg}
\end{figure}

\section{Rational approximation for the general problems of sums of fractional operators} \label{sec:ra}
Let $s, t \in [-1, 1]$ and $\alpha, \beta \geq 0$ where at least one of $\alpha, \beta$ is not zero. For interval $I \subset \mathbb{R}^+$, consider a function
\( 
	f(x) = (\alpha x^s + \beta x^t)^{-1}, \quad x \in I.
\).
The basic idea is to find a rational function $R(x)$ approximating $f$ on $I$, that is,
\(
	R(x) = \frac{P_{k'}(x)}{Q_{k}(x)}\approx f(x)
\),
where $P_{k'}$ and $Q_{k}$ are polynomials of degree $k'$ and $k$, respectively. Assuming $ k' \leq k $, the rational function can be given in the following partial fraction form
\begin{equation*} 
	R(x) = c_0 + \sum_{i=1}^N \frac{c_i}{x - p_i},
\end{equation*}
for $ c_0 \in \RR $, $ c_i, p_i \in \CC $, $ i = 1, 2, \dots, N$. The coefficients $p_i$ and $c_i$ are called \emph{poles} and \emph{residues} of the rational approximation, respectively.

We note that the rational approximation has been predominantly explored to approximate functions with only one fractional power, that is $x^{-\bar{s}}$ for $\bar{s} \in (0, 1) $ and $ x > 0 $. Additionally, the choice of the rational approximation method that computes the poles and residues is not unique. One possibility is the BURA method which first computes the best uniform rational approximation $\bar{r}_{\beta}(x)$ of $x^{\beta -\bar{s}}$ for a positive integer $\beta > \bar{s}$ and then uses $\bar{r}(x) = \frac{\bar{r}_{\beta}(x)}{x}$ to approximate $x^{-\bar{s}}$. Another possible choice is to use the rational interpolation of $z^{-\bar{s}}$ to obtain $\bar{r}(x)$. The AAA algorithm proposed in \cite{nakatsukasaAAAAlgorithmRational2018} is a good candidate. The AAA method is based on the representation of the rational approximation in barycentric form and greedy selection of the interpolation points. Both approaches lead to the poles $ p_i \in \RR, \, p_i \leq 0 $ for the case of one fractional power. An overview of rational approximation methods can be found in \cite{hofreitherUnifiedViewNumerical2020}. 

The location of the poles is crucial in rational approximation preconditioning. For $\bar{f}(x) = x^{\bar{s}}$, $\bar{s} \in (0,1)$ the poles of the rational approximation for $\bar{f}$ are all real and negative. Hence, in the case of a positive definite operator $\mathcal{D}$, the approximation of $\mathcal{D}^{\bar{s}}$ requires only inversion of positive definite operators of the form $\mathcal{D}+|p_i| I$, for $i = 1, 2, \dots, N$, $p_i \neq 0 $. Such a result for rational approximation of $\bar{f}$ with $\bar{s}\in(-1,1)$ is found in a paper by H.~Stahl~\cite{2003Stahl-a}.  In the following, we present an extensive set of numerical tests for the class of \emph{sum of fractional operators}, which gives a wide class of efficient preconditioners for the multiphysics problems coupled through an interface. The numerical tests show that the poles remain real and nonpositive in most combinations of fractional exponents $s$ and $t$.

Let $ V $ be a Hilbert space and $ V'$ be its dual. Consider a symmetric positive definite (SPD) operator $ A : V \to V' $. Then, the rational function $ R(\cdot) $ can be used to approximate $ f(A) $ as follows,
\begin{equation*}
	z = f(A) r \approx c_0 r + \sum_{i=1}^N c_i \left( A - p_i I \right)^{-1}r
\end{equation*}
with $ z \in V $ and $r \in V'$. The overall algorithm is shown in \Cref{alg:rational-approx}.
\begin{algorithm}
	\caption{Compute $ z = f(A) r $ using rational approximation.} \label{alg:rational-approx}
	\begin{algorithmic}[1]
		\STATE Solve for $ w_i $:
		$ \left( A - p_i I \right) w_i = r, \quad i = 1, 2, \dots, N. $
		\STATE Compute: $ z = c_0 r + \sum\limits_{i=1}^N c_i w_i $
	\end{algorithmic}
\end{algorithm}

Without loss of generality we let the operator $ A $ be a discretization of the Laplacian operator $ -\Delta $,
and $ I $ is the identity defined using the standard $L^2$ inner product. In particular, unlike in \eqref{eq:darcy_stokes_iface}
we assume that $ -\Delta $ is SPD (e.g. by imposing boundary conditions eliminating the constant nullspace).
Therefore, the equations in Step 1 of \Cref{alg:rational-approx} can be viewed as discretizations
of the shifted Laplacian problems $ -\Delta \, w_i - p_i \, w_i  = r $, $p_i<0$, and we can use
efficient numerical methods, such as Algebraic MultiGrid (AMG) methods \cite{1stAMG,2017XuZikatanov}, for their solution.

We would like to point out that in the implementation, the operators involved in \Cref{alg:rational-approx} are replaced by their matrix representations on a concrete basis and are properly scaled. We address this in more detail in the following section.

\subsection{Preconditioning} \label{subsec:ra_preconditioning}

Let $ \sfA $ be the stiffness matrix associated with $ -\Delta$ and $\sfM$ a corresponding mass matrix of the $L^2$ inner product. Also, denote with $ n_c $ the number of columns of $ \sfA $. The problem we are interested in is constructing an efficient preconditioner for the solution of the linear system $F(\sfA) \mathsf{x} = \mathsf{b}$.
Thus, we would like to approximate $ f(\sfA) = F(\sfA)^{-1} $ using the rational approximation $ R(x) $ of $ f(x)=\frac{1}{F(x)} $. 

Let $ \sfI $ be a $ n_c \times n_c $ identity matrix and let $ \sfU $ be an $ \sfM $-orthogonal matrix of the eigenvectors of the generalized eigenvalue problem $ \sfA \mathsf{u}_j = \lambda_j \sfM \mathsf{u}_j $, $ j = 1, 2, \dots, n_c $, namely,
\begin{equation}\label{eq:gep_ra}
	\sfA \sfU = \sfM \sfU \sfLam, \quad \sfU^T \sfM \sfU = \sfI \quad \Longrightarrow\quad \sfU^T \sfA \sfU = \sfLam.
\end{equation}
For any continuous function $ G : [0, \rho] \rightarrow \mathbb{R}$ we define
\begin{equation}\label{eq:function_gep}
	G(\sfA) \coloneqq \sfM \sfU G(\sfLam) \sfU^T \sfM, \quad\mbox{i.e.}, \quad
 	f(\sfA)=\sfM \sfU f(\sfLam) \sfU^T \sfM.
\end{equation}
where $ \rho \coloneqq \rho \left(\sfM^{-1} \sfA \right) $ is the spectral radius of the matrix $ \sfM^{-1} \sfA $. 

A simple consequence from the Chebyshev Alternation Theorem is that the residues/poles for the rational approximation of $f(x)$ for $x\in [0,\rho]$ are obtained by scaling the residues/poles of the rational approximation of $ g(y) = f\left( \rho y \right) $ for $ y \in [0,1] $. Indeed, we have, $c_i(f)=\rho c_i(g)$, and $p_i(f)=\rho p_i(g)$. Therefore, in the implementation, we need an upper bound on $ \rho \left( \sfM^{-1}\sfA \right) $. For $\mathbb{P}_1 $ finite elements, such a bound can be obtained following the arguments from \cite{wathen1987realistic},
\begin{equation}\label{diag}
	\rho\left(\sfM^{-1} \sfA\right)
	\le \frac{1}{\lambda_{\min{}}(\sfM)} \|\sfA\|_{\infty}
	=d(d+1)\left\|\operatorname{diag}(\sfM)^{-1}\right\|_{\infty} \| \sfA \|_{\infty}
\end{equation}
with $d$ the spatial dimension\footnote{Such estimates can be carried out for Lagrange finite elements of any polynomial degree because the local mass matrices are of a special type: a constant matrix, which depends only on the dimension and the polynomial degree, times the volume of the element.}.

\begin{proposition} \label{prop:ra_for_matrix}
	Let $ R_f(\cdot) $ be the rational approximation for the function $ f(\cdot) $ on $ [0, \rho] $. Then for the stiffness matrix $ \sfA $ and mass matrix $ \sfM $ satisfying \eqref{eq:gep_ra}, we have
		\begin{equation} \label{eq:ra_for_matrix}
			f(\sfA) \approx R_f(\sfA) = c_0 \sfM^{-1} + \sum_{i=1}^N c_i \left(\sfA - p_i \sfM \right)^{-1}.
		\end{equation}
\end{proposition}
\begin{proof}
The relations \eqref{eq:gep_ra}	imply that $ \sfU \sfU^T = \sfM^{-1}$, and therefore, 
\begin{equation*} 
			\sfA \sfU - p_i \sfM \sfU = \sfM \sfU (\sfLam - p_i \sfI)
      \Longleftrightarrow 
			(\sfA - p_i \sfM)\sfU = \sfM \sfU(\sfLam - p_i \sfI).
   \end{equation*}
   This is equivalent to \(
   (\sfLam - p_i \sfI)^{-1}\underbrace{\sfU^{-1} \sfM^{-1}}_{\sfU^T}  = \sfU^{-1}(\sfA - p_i \sfM)^{-1}
   \). Hence, 
   \begin{equation*}
			(\sfLam - p_i \sfI)^{-1} = \sfU^{-1}(\sfA - p_i \sfM)^{-1} \sfU^{-T}.
\end{equation*}
A straightforward substitution in~\eqref{eq:function_gep} then shows \eqref{eq:ra_for_matrix}.
	\qed
\end{proof}

In addition, to apply the rational approximation preconditioner, we need to compute the actions of $\sfM^{-1}$ and each $\left(\sfA - \rho p_i \sfM \right)^{-1}$. If $ p_i \in \RR, \, p_i \leq 0$, this leads to solving a series of elliptic problems where the AMG methods are very efficient. 

\section{Numerical results} \label{sec:results}
In this section, we present two sets of experiments: (1) on the robustness of the rational approximation with respect to the scaling parameters and the fractional exponents; and (2) on the efficacy of the preconditioned minimal residual (MinRes) method as a solver for Darcy--Stokes coupled model. We use the AAA algorithm \cite{nakatsukasaAAAAlgorithmRational2018} to construct a rational approximation. The discretization and solver tools are \verb|Python| modules provided by FEniCS\textunderscore{ii} \cite{kuchta2021assembly}, cbc.block \cite{mardal2012block}, and interfaced with the HAZmath library \cite{hazmath}.

\subsection{Approximating the sum of two fractional exponents} \label{subsec:example_1}

In this example, we test the approximation power of the rational approximation computed by AAA algorithm regarding different fractional exponents $s, t$ and parameters $\alpha, \beta$. That is, we study the number of poles $N $ required to achieve
\begin{equation*}
	\| f - R \|_{\infty} = \max\limits_{x \in [0, 1]} \left| (\alpha x^s + \beta x^t)^{-1} - \left(c_0 + \sum_{i=1}^N \frac{c_i}{x - p_i} \right) \right| \leq \epsilon_{\text{RA}},
\end{equation*}
for a fixed tolerance $ \epsilon_{\text{RA}} = 10^{-12}$. In this case, we consider the fractional function $ f $ to be defined on the unit interval $I = (0, 1]$. As we noted earlier, however, the approximation can be straightforwardly extended to any interval. We also consider the scaling regarding the magnitude of parameters $ \alpha $ and $ \beta $. Specifically, in case when $ \alpha > \beta $, we rescale the problem with $ \gamma_\alpha = \frac{\beta}{\alpha} < 1 $ and approximate
\begin{equation*}
	\widetilde{f}(x) = (x^s + \gamma_\alpha x^t)^{-1} \approx R_{\widetilde{f}}(x).
\end{equation*}
Then the rational approximation for the original function $ f $ is given as
\begin{equation*}
	f(x) \approx R(x) = \frac{1}{\alpha} R_{\widetilde{f}}(x).
\end{equation*}
Similar can be done in case when $ \beta > \alpha $.

\begin{figure}[t!]
    \centering
    \includegraphics[width=\textwidth]{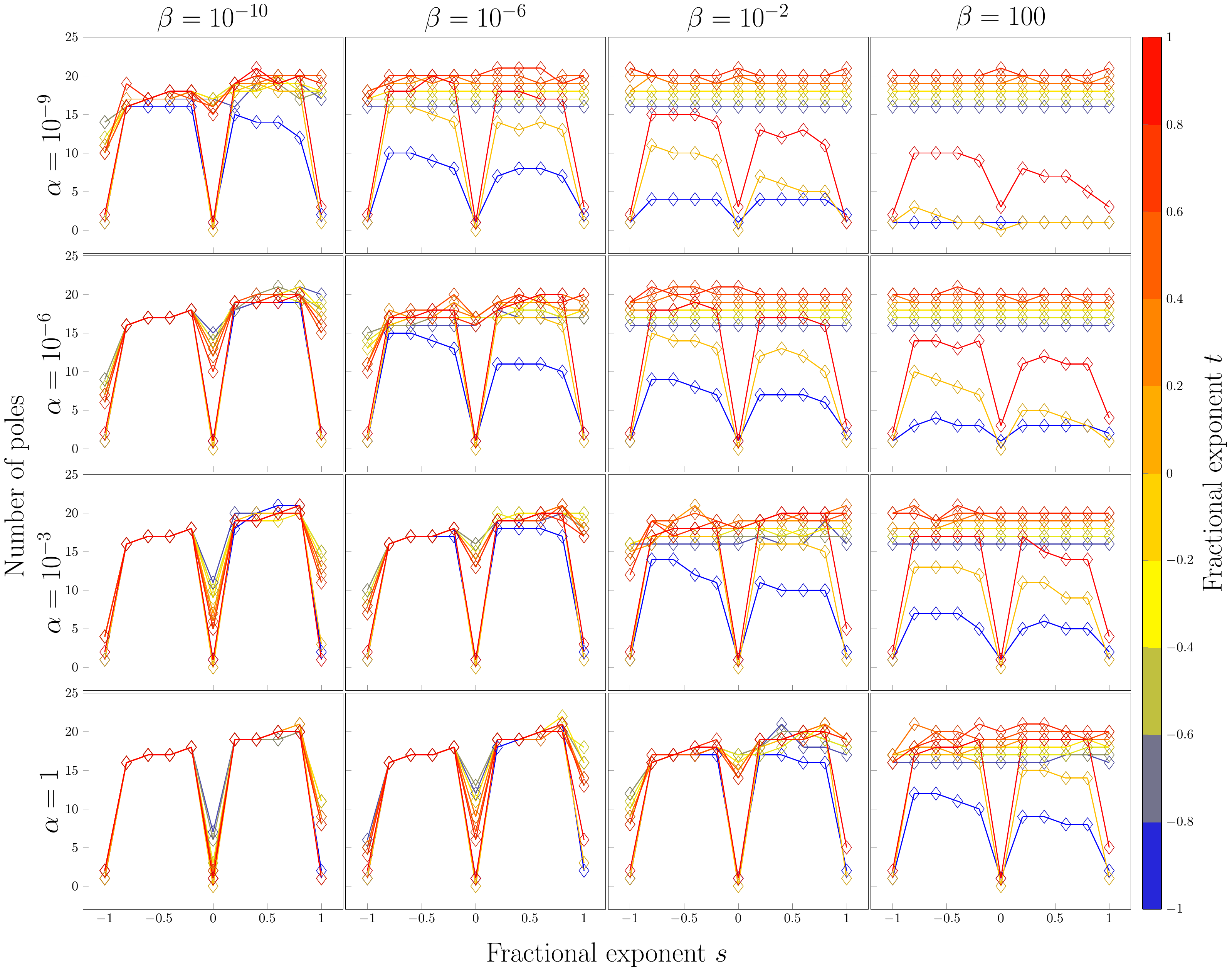}
    \vspace{-25pt}
    \caption{Visualization of the number of poles in the rational approximations of the function $f(x) = (\alpha x^s + \beta x^t)^{-1}$ for $ x \in(0, 1] $ with regards to varying fractional exponents $s, t$ and coefficients $\alpha, \beta$.}
    \label{fig:number_of_poles}
 \end{figure}

The results are summarized in \Cref{fig:number_of_poles}. To obtain different parameter ratios $\gamma_\alpha$, we take $\alpha \in \{10^{-9}, 10^{-6}, 10^{-3}, 1 \} $ and $ \beta \in \{10^{-10}, 10^{-6}, 10^{-2}, 10^2 \} $. Furthermore, we vary the fractional exponents $s, t \in [-1, 1] $ with the step 0.2. We observe that the number of poles $N$ remains relatively uniform with varying the exponents, except in generic cases when $ s, t = \{-1, 0, 1 \}$. For example, for the combination $(s, t) = (1, 1)$, the function we are approximating is $ f(x) = \tfrac{1}{2x} $, thus the rational approximation should return only one pole $ p_1 = 0 $ and residues $ c_0 = 0, c_1 = \half $. We also observe that for the fixed tolerance of $ \epsilon_{\text{RA}} = 10^{-12}$, we obtain a maximum of 22 poles in all cases.

Additionally, we remark that in most test cases, we retain real and negative poles, which is a desirable property to apply the rational approximation as a positive definite preconditioner. However, depending on the choice of fractional exponents $s, t$ and tolerance $\epsilon_{\text{RA}}$, the algorithm can produce positive or a pair of complex conjugate poles. Nevertheless, these cases are rare, and the number and the values of those poles are small. Therefore numerically, we do not observe any significant influence on the rational approximation preconditioner. More concrete analytical results on the location of poles for sums of fractionalities are part of our future research.

\subsection{The Darcy--Stokes problem}\label{subsec:example_ds}
In this section, we discuss the performance of RA approximation of $S^{-1}$ in the Darcy--Stokes preconditioner \eqref{eq:darcy_stokes_precond}. To this end, we consider $\Omega_S=(0, 1)^2$, $\Omega_D=(\tfrac{1}{4}, \tfrac{3}{4})^2$ and fix the value of Beavers-Joseph-Saffman parameter $\alpha=1$ while permeability and viscosity  are varied\footnote{These ranges are identified as relevant for many applications in biomechanics \cite{boon2022robust}.} $10^{-6}\leq K\leq 1$ and $10^{-6}\leq\mu\leq 10^{2}$. The system is discretized by $\mathbb{BDM}_k$-$\mathbb{BDM}_k$-$\mathbb{P}^{\text{disc}}_{k-1}$-$\mathbb{P}^{\text{disc}}_{k-1}$-$\mathbb{P}^{\text{disc}}_k$ elements, $k=1, 2, 3$, which were shown to provide convergent approximations in \Cref{fig:darcy_stokes_cvrg}. A hierarchy of meshes $\Omega^h_S$, $\Omega^h_D$ is obtained by uniform refinement. We remark that the stabilization constants $\gamma$ in \eqref{eq:sym_grad_BDM} and \eqref{eq:delta_dg} are chosen as $\gamma=20k$ and $\gamma=10k$ respectively. The resulting linear systems are then solved by preconditioned MinRes method. To focus on the RA algorithm in the preconditioner, for all but the multiplier block, an exact LU decomposition is used. Finally, the iterative solver is always started from a zero initial guess and terminates when the preconditioned residual norm drops below $10^{-10}$.

We present two sets of experiments. First, we fix the tolerance in the RA algorithm at $\epsilon_{\text{RA}}=2^{-40}\approx 10^{-12}$ and demonstrate that using RA in the preconditioner \eqref{eq:darcy_stokes_precond} leads to stable MinRes iterations for any practical values of the material parameters, the mesh resolution $h$ and the polynomial degree $k$ in the finite element discretization (see \Cref{fig:darcy_stokes_param_robust}). Here it can be seen that the number of iterations required for convergence is bounded in the above-listed quantities. Results for different polynomial degrees are largely similar and appear to be mostly controlled by material parameters. However, despite the values of $K$, $\mu$ spanning several orders of magnitude, the iterations vary only between 30 to 100.

\begin{figure}[ht!]
  \centering
  \includegraphics[width=\textwidth]{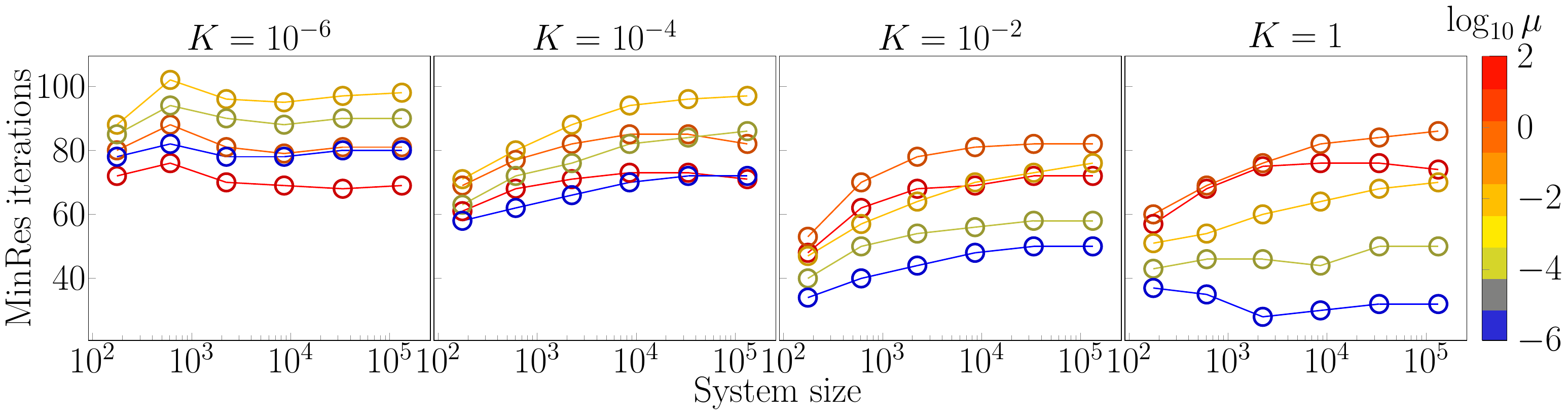}
  \includegraphics[width=\textwidth]{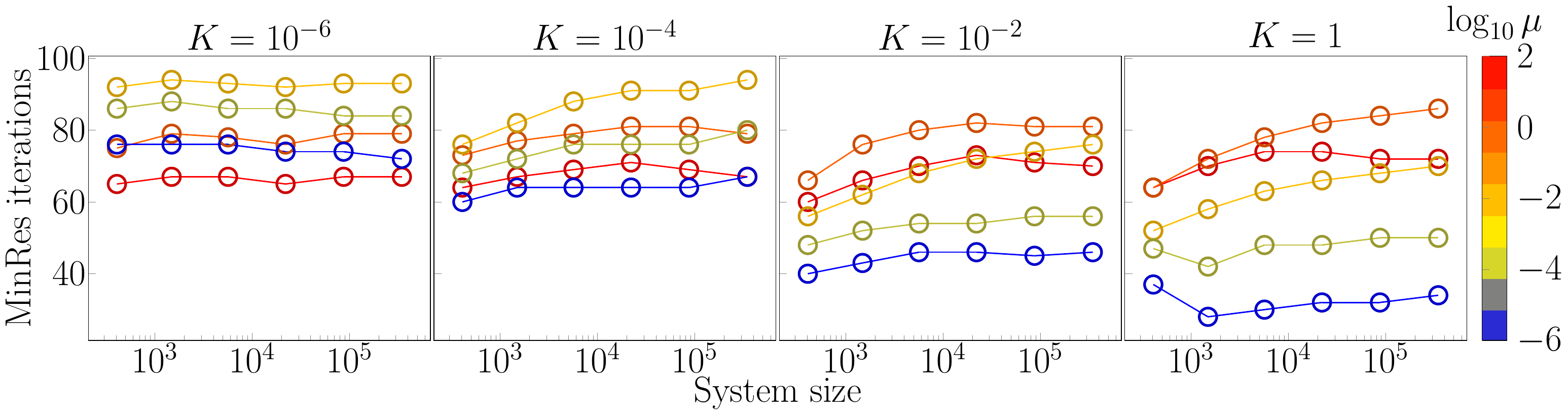}
  \includegraphics[width=\textwidth]{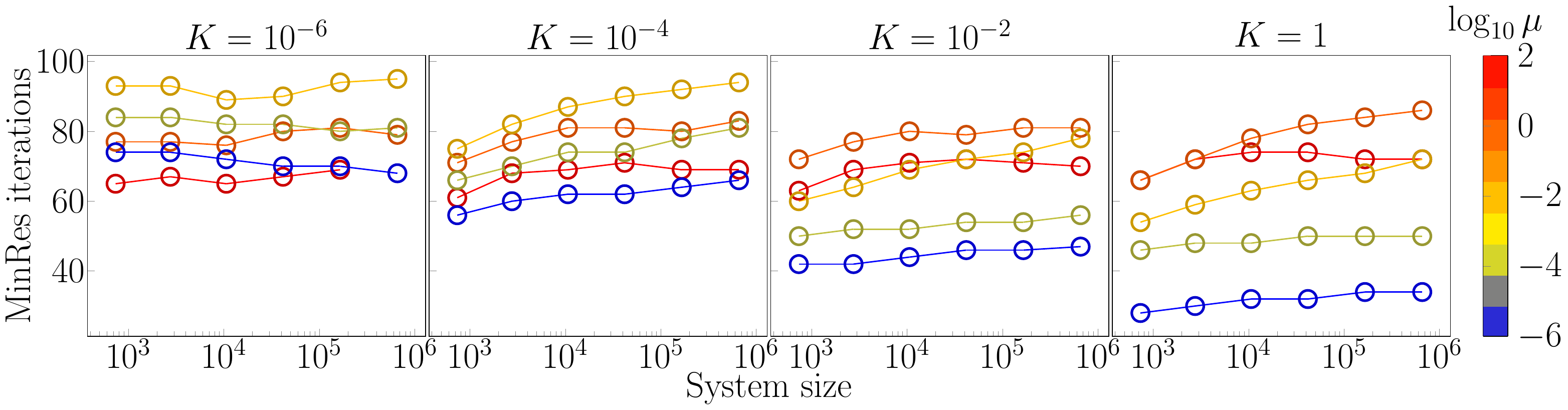}
  \vspace{-25pt}
  \caption{
    Number of MinRes iterations required for convergence with preconditioner \eqref{eq:darcy_stokes_precond} using the RA with tolerance $\epsilon_{\text{RA}}=2^{-40}$. Setup from \Cref{fig:darcy_stokes_cvrg} is considered with the system discretized by  $\mathbb{BDM}_k$-$\mathbb{BDM}_k$-$\mathbb{P}^{\text{disc}}_{k-1}$-$\mathbb{P}^{\text{disc}}_{k-1}$-$\mathbb{P}^{\text{disc}}_k$
    elements. (Top) $k=1$, (middle) $k=2$, (bottom) $k=3$.
  }
  \label{fig:darcy_stokes_param_robust}
\end{figure}

Next, we assess the effects of the accuracy in RA on the performance of preconditioned MinRes solver. Let us fix $k=1$ and vary the material parameters as well as the RA tolerance $\epsilon_{\text{RA}}$. In \Cref{fig:darcy_stokes_ra_robust}, we observe that the effect of $\epsilon_{\text{RA}}$ varies with material properties (which enter the RA algorithm through scaling). In particular, it can be seen that for $K=1$, the number of MinRes iterations is practically constant for any $\epsilon_{\text{RA}}\leq 10^{-1}$. On the other hand,  when $K=10^{-6}$, the counts vary with $\epsilon_{\text{RA}}$ and to a lesser extent with $\mu$. Here, lower accuracy typically leads to a larger number of MinRes iterations. However, for $\epsilon_{\text{RA}}\leq 10^{-4}$ the iterations behave similarly. We remark that with $K=10^{-6}$ (and any $10^{-6}\leq\mu\leq 10^{2}$), the tolerance $\epsilon_{\text{RA}}=10^{-1}$ leads to 2 poles, cf. \Cref{fig:darcy_stokes_ra_perf},  while for $K=1$ there are at least 5 poles needed in the RA approximation.

\begin{figure}
  \centering
  \includegraphics[width=\textwidth]{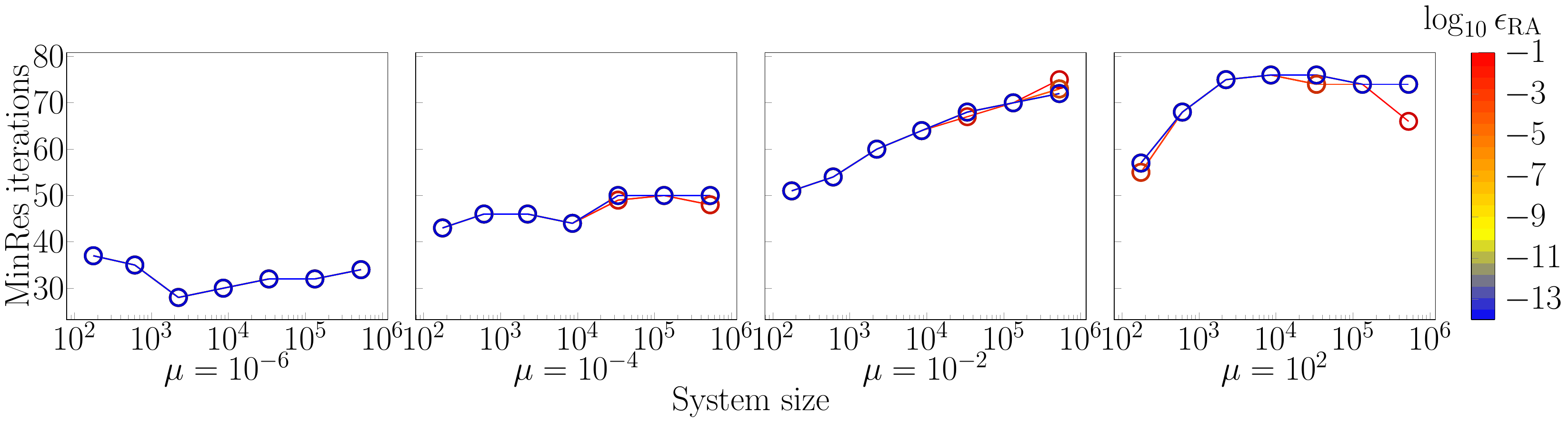}
  \includegraphics[width=\textwidth]{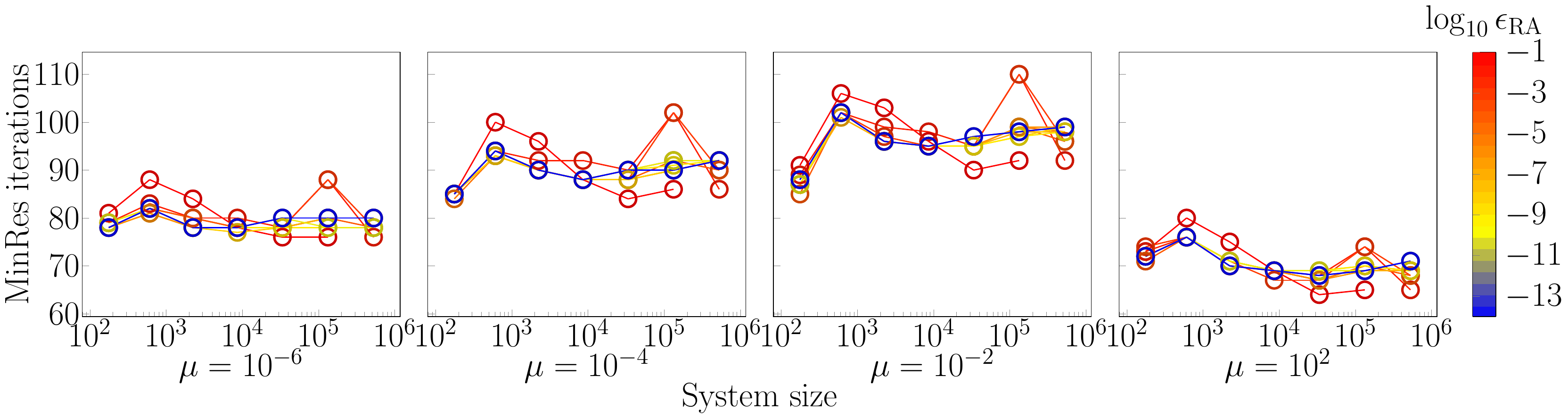}
  \vspace{-25pt}  
  \caption{
    Number of MinRes iterations required for convergence with preconditioner
    \eqref{eq:darcy_stokes_precond} using the RA with varying tolerance $\epsilon_{\text{RA}}$.
    (Top) $K=1$. (Bottom) $K=10^{-6}$. Setup as in \Cref{fig:darcy_stokes_param_robust} is used with discretization by
    $\mathbb{BDM}_1$-$\mathbb{BDM}_1$-$\mathbb{P}^{\text{disc}}_{0}$-$\mathbb{P}^{\text{disc}}_{0}$-$\mathbb{P}^{\text{disc}}_1$ elements.
  }
  \label{fig:darcy_stokes_ra_robust}
\end{figure}

Our results demonstrate that RA approximates $S^{-1}$ in \eqref{eq:darcy_stokes_iface}, which leads to a robust, mesh, and parameter-independent Darcy--Stokes solver. We remark that though the algorithm complexity is expected to scale with the number of degrees of freedom on the interface, $n_{h}=\text{dim}\Lambda_h$, which is often considerably smaller than the total problem size, the setup cost may become prohibitive. This is particularly true for spectral realization, which often results in $\mathcal{O}(n^3_h)$ complexity. To address such issues, we consider how the setup time of RA and the solution time of the MinRes solver depend on the problem size. 

In \Cref{fig:darcy_stokes_ra_perf} we show the setup time of RA (for fixed material parameters) as function of mesh size and $\epsilon_{\text{RA}}$. It can be seen that the times are $<0.1\,\text{s}$ and practically constant with $h$ (and $n_{h}$). As with the number of poles, the small variations in the timings with $h$ are likely due to different scaling of the matrices $\sfA$ and $\sfM$. We note that in our experiments $32\leq n_h\leq 1024$. Moreover, since $\Lambda_h$ is in our experiments constructed from $\mathbb{P}^{\text{disc}}_1$ we apply the estimate \eqref{diag}. 

In \Cref{fig:darcy_stokes_ra_perf}, we finally plot the dependence of the solution time of the preconditioned MinRes solver on the problem size. Indeed we observe that the solver is of linear complexity. In particular, application of rational approximation of $S^{-1}$ in our implementation requires $\mathcal{O}(n_h)$ operations. We remark that here the solvers for the shifted Laplacian problems are realized by the conjugate gradient method with AMG as a preconditioner. Let us also recall that the remaining blocks of the preconditioner are realized by LU, where the setup cost is not included in our timings. However, LU does not define efficient preconditioners for the respective blocks. Here instead, multilevel methods could provide order optimality, and this is a topic of current and future research.
\begin{figure}
  \centering
  \includegraphics[width=0.325\textwidth, height=0.15\textheight]{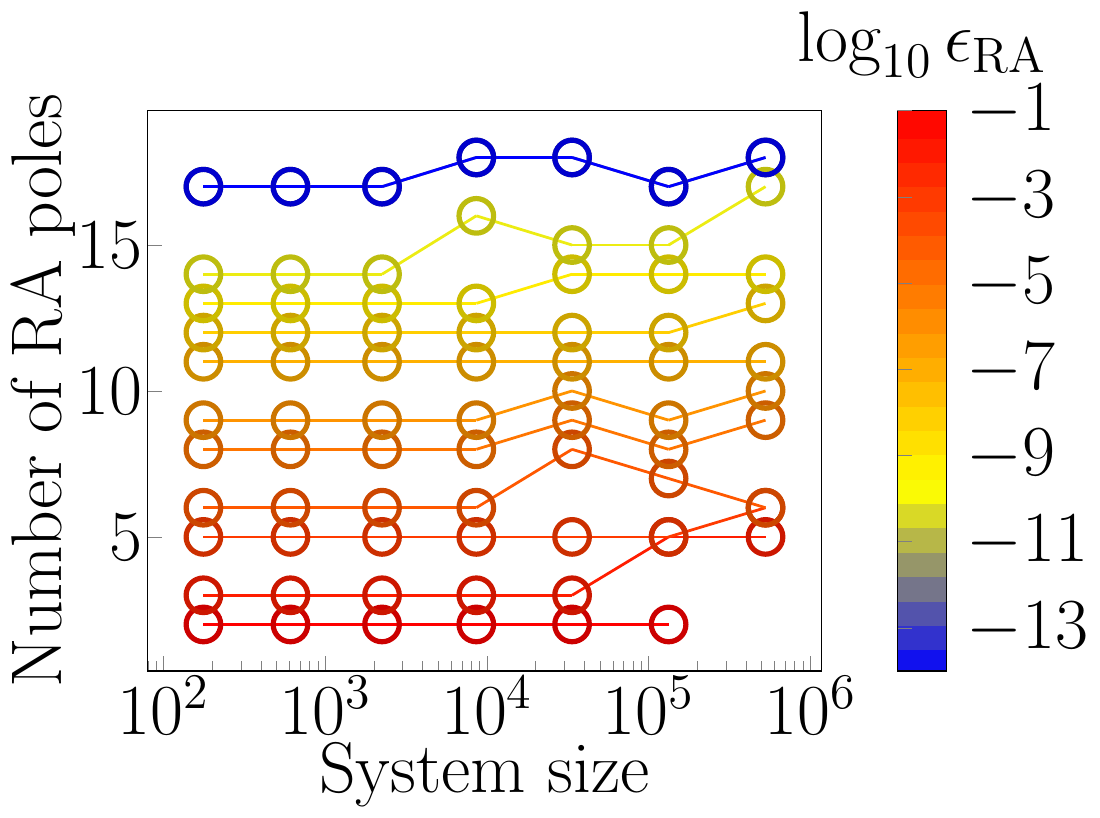}
  \includegraphics[width=0.325\textwidth, height=0.15\textheight]{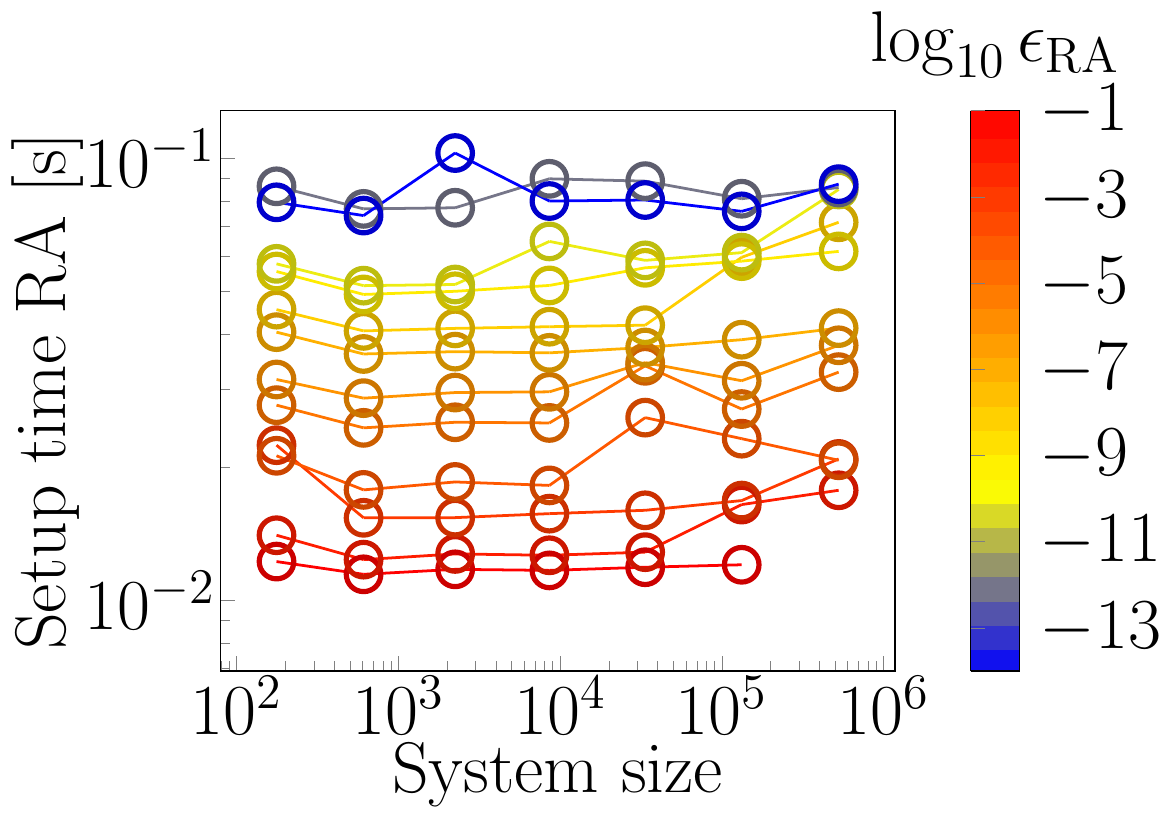}  
  \includegraphics[width=0.325\textwidth, height=0.15\textheight]{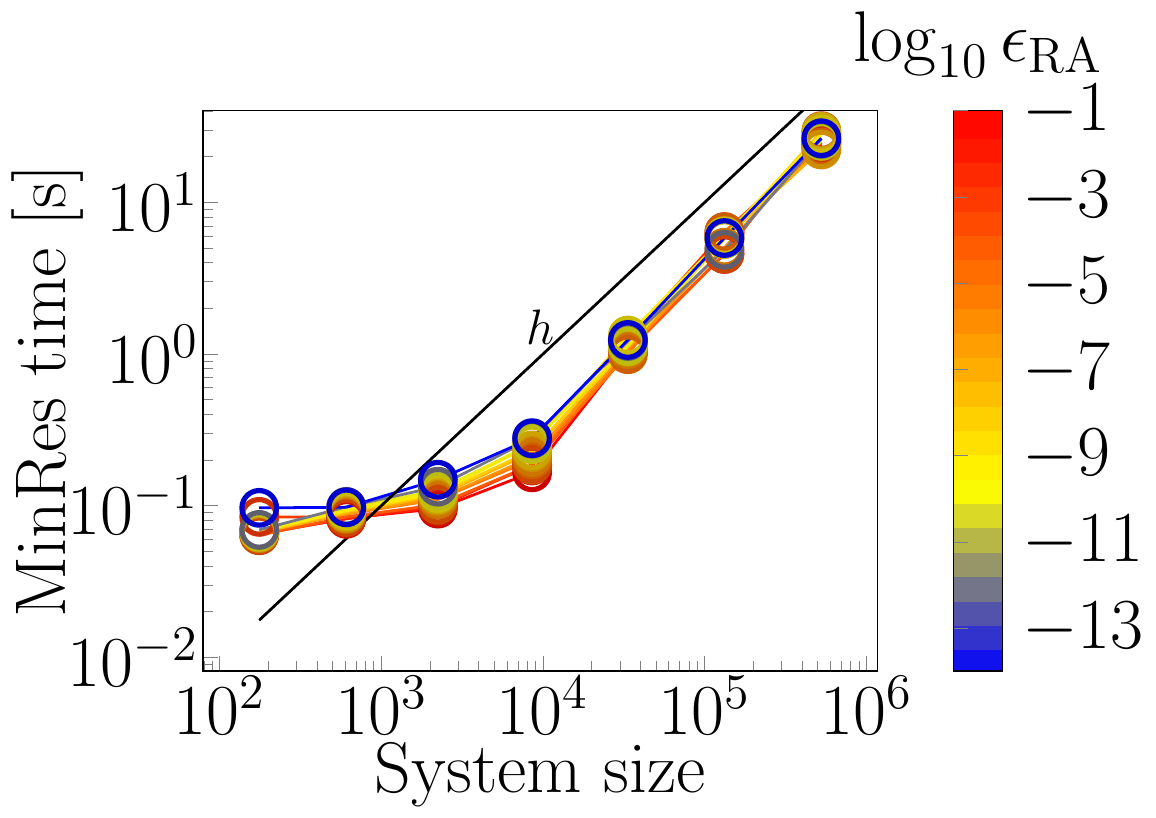}
  \vspace{-10pt}
  \caption{
    Dependence of the number of poles (left), the setup time of RA (center) and runtime of the MinRes solver (right) on mesh size $h$ and RA tolerance $\epsilon_{\text{RA}}$. Parameters in the Darcy--Stokes problem are fixed at $K=10^{-6}$, $\mu=10^{-2}$ and $\alpha=1$. Setup as in \Cref{fig:darcy_stokes_param_robust} is used with discretization by $\mathbb{BDM}_1$-$\mathbb{BDM}_1$-$\mathbb{P}^{\text{disc}}_{0}$-$\mathbb{P}^{\text{disc}}_{0}$-$\mathbb{P}^{\text{disc}}_1$ elements.    
  }
  \label{fig:darcy_stokes_ra_perf}
\end{figure}

\section{Conclusions}\label{sec:conclusions}
We have demonstrated that RA provides order optimal preconditioners for sums of fractional powers of SPD operators and can thus be utilized to construct parameter robust and order optimal preconditioners for multiphysics problems. The results are of practical interest for constructing efficient preconditioning on interfaces in models for which fractional weighted Sobolev spaces are the natural setting for the resulting differential operators, for example, when flow interacts with porous media. The techniques presented here could aid the numerical simulations in a wide range of biology, medicine, and engineering applications.   

\section{Acknowledgments} 
A. Budi\v{s}a, M. Kuchta, K.~A. Mardal, and L.~Zikatanov acknowledge the financial support from the SciML project funded by the Norwegian Research Council grant 102155. The work of X. Hu is partially supported by the National Science Foundation under grant DMS-2208267. M.~Kuchta acknowledges support from the Norwegian Research Council grant 303362. The work of L.~Zikatanov is supported in part by the U.S.--Norway Fulbright Foundation and the U.S. National Science Foundation grant DMS-2208249. The collaborative efforts of Hu and Zikatanov were supported in part by the NSF DMS-2132710 through the Workshop on Numerical Modeling with Neural Networks, Learning, and Multilevel FE.

\bibliographystyle{splncs04}
\bibliography{references}

\end{document}